\documentclass[12pt,reqno]{amsart}

\usepackage[utf8]{inputenc}
\usepackage[T1]{fontenc}

\usepackage{amsmath, amssymb, amsthm}
\usepackage{geometry}
\usepackage{setspace}
\usepackage{enumitem}
\usepackage{hyperref}
\hypersetup{
  colorlinks,
  citecolor=blue,
  linkcolor=blue,
  urlcolor=green}
\usepackage{marginnote}
\usepackage{xcolor}
\usepackage{upgreek}
\usepackage{tikz}
\usetikzlibrary{calc}
\usepgflibrary {shadings} 
\usetikzlibrary{intersections}
\usetikzlibrary{patterns}
\usepackage{mathtools}
\usepackage{verbatim}
\usepackage{lipsum}  
\usepackage{mathrsfs}
\usepackage{stmaryrd}
\usepackage[
    colorinlistoftodos,
    backgroundcolor=green!20!white,
    bordercolor=green!75!black,
    linecolor=green!75!black
]{todonotes}
\usepackage{cleveref}
\usepackage{bm}

\numberwithin{equation}{section}

\theoremstyle{plain}
\newtheorem{theorem}{Theorem}
\newtheorem{lemma}{Lemma}[section]

\newtheorem{corollary}{Corollary}

\theoremstyle{definition}

\newcommand{\C}{\mathbf{C}}
\newcommand{\R}{\mathbf{R}}
\newcommand{\Q}{\mathbf{Q}}

\newcommand{\Sp}{\mathbf{S}}

\renewcommand{\div}{\operatorname{div}}

\newcommand{\G}{G}

\renewcommand{\b}[1]{\bm{#1}}

\newcommand{\eps}{\varepsilon}
\begin{document}

\title{Stable solutions in the abelian Higgs model}
\author{Marco Badran}
\address{ETH Z\"urich, Department of Mathematics, Rämistrasse 101, 8092 Zürich, Switzerland.}
 	\email{marco.badran@math.ethz.ch}
 \address{Bocconi University, Via Roentgen 1, 20136 Milan, Italy}
 	\email{marco.badran@unibocconi.it}

\begin{abstract}
	We classify entire stable solutions with quadratic energy growth in the $4$-dimensional abelian Higgs model. 
	More precisely, we prove that such solutions are holomorphic with respect to some complex structure, namely they satisfy the first order equations introduced by Bradlow, and consequently their nodal sets are holomorphic curves. Conversely, we construct solutions whose nodal set is any prescribed holomorphic curve with quadratic area growth. As a corollary, we show that all stable solutions with linear energy growth in $\R^3$ are two-dimensional Bogomolnyi vortices.
\end{abstract}

\maketitle


\section{Introduction}

The abelian Higgs model is a geometric variational model arising from superconductivity and relativistic gauge field theory.
Given a Riemannian manifold $(M,g)$ and a complex line bundle $L\to M$ endowed with a Hermitian structure $\langle\cdot,\cdot\rangle$, we define the self-dual Yang--Mills--Higgs energy as 
\begin{equation}\label{eq: energy}
	E(u,\nabla)\coloneqq \frac12\int e(u,\nabla),\qquad e(u,\nabla)\coloneqq |\nabla u|^2+|F_\nabla|^2+\frac14(1-|u|^2)^2.
\end{equation}
where $\nabla$ is a metric connection on $L$, $u\in \Gamma(L)$ is a section and $F_\nabla\coloneqq i\nabla^2$ denotes the (real) curvature of the connection. Critical points solve 
\begin{equation*}
	\nabla^*\nabla u=\frac12(1-|u|^2)u,\quad d^*F_\nabla =\langle \nabla u,iu\rangle.
\end{equation*}
In recent years, the abelian Higgs model has received considerable attention because of its intimate relationship to the codimension two volume functional \cite{Pigati-Stern2021,Parise-Pigati-Stern2024a,Parise-Pigati-Stern2024b,Badran-delPino2023,Badran-delPino2024,DePhilippis-Pigati2024,DePhilippis-Halavati-Pigati2024,Nguyen-Wang2026}. Critical points of the rescaled energies
\begin{equation*}
	E_\eps(u,\nabla)\coloneqq \frac12\int |\nabla u|^2+\eps^2|F_\nabla|^2+\frac1{4\eps^2}(1-|u|^2)^2
\end{equation*}
act as a diffuse approximation of (generalised) minimal surfaces in codimension two \cite{Pigati-Stern2021,Parise-Pigati-Stern2024a,Parise-Pigati-Stern2024b}. This perspective is appealing because a variational existence and regularity theory for the volume functional is notoriously challenging, while diffuse models are generally more flexible. Of particular success is the codimension-one setting, where the diffuse Allen--Cahn model was used to recover classical results, such as the existence of embedded minimal hypersurfaces in any closed $n+1$-manifold for $3\leq n+1\leq 7$ \cite{Guaraco2018} or \emph{Yau's conjecture} on the existence of infinitely many embedded minimal surfaces in $3$-manifolds for generic metrics \cite{Chodosh-Mantoulidis2020}. The regularity of the limit interfaces follows two possible routes: either using a powerful structural theorem for limiting varifolds as in \cite{Guaraco2018}, building on the deep regularity theory of Tonegawa and Wickramasekera \cite{Tonegawa-Wickramasekera2012,Wickramasekera2014}, or proving directly uniform estimates for the level sets of solutions, as in \cite{Chodosh-Mantoulidis2020}, building on the theory of \cite{Wang-Wei2019a,Wang-Wei2019b}. A crucial ingredient in the latter strategy is the so-called \emph{stable De Giorgi conjecture (with energy bounds)}, that is the classification of entire stable solutions of the Allen--Cahn equation with Euclidean energy growth, which are known to be one-dimensional (i.e.~ with flat level sets) in ambient dimension $n+1=2,3$ \cite{Ghoussoub-Gui1998,Ambrosio-Cabre2000} and more recently in $n+1=4$ \cite{Florit-Serra2025}.

\medskip

In codimension two no structural regularity theorem for Yang--Mills--Higgs limits comparable to \cite{Tonegawa-Wickramasekera2012,Wickramasekera2014} is known, motivating an alternative approach based on the classification of stable solutions. Unfortunately, there is no hope for solutions to have flat level sets by the $U(1)$-gauge invariance of the energy: for any $\Sp^1$-valued map $\zeta$ 
\begin{equation*}
	E_\eps(u,\nabla)=E_\eps(\zeta u,\nabla-i\zeta^*(d\theta))
\end{equation*}
which makes the phase of a solution completely arbitrary. Even restricting our attention to the nodal set $\{u=0\}$ or fixing a convenient gauge, non-flat stable counterexamples exist \cite{Liu-Wei-Ye2024,Liu-Ma-Wei-Wu2025}.
This is to be expected, since there are many examples of non-flat stable minimal surfaces in codimension two: all holomorphic curves, which are in fact calibrated and hence area minimising. In the beautiful work \cite{Micallef1984}, Micallef proved that these examples are the only stable ones, in the sense that any complete, connected, parabolic, orientable, stable,  minimal immersion $M^2\to \R^4$ is holomorphic with respect to some complex structure. Under the Euclidean volume growth assumption, such surfaces are in fact algebraic \cite{Edelen-Reyna-Minter2026}.

\medskip

There is a natural notion of holomorphicity for Yang--Mills--Higgs solutions in four dimensions, introduced by Bradlow \cite{Bradlow1990} in the context of K\"ahler surfaces. Indeed on a K\"ahler surface $(M^2,g,J,\omega)$ the energy  density in \eqref{eq: energy} decomposes as 
\begin{equation}\label{eq: bradlow decomp}
	\begin{split}
		e(u,\nabla)d\mathrm{vol}_g&=\left(2|(\nabla u)^{0,1}_J|^2+4|F_{\nabla;J}^{0,2}|^2+\left(\Lambda_JF_\nabla-h(u)\right)^2\right)d\mathrm{vol}_g\\
		&+F_\nabla\wedge \omega-F_\nabla\wedge F_\nabla+(\div V_J)d\mathrm{vol}_g
	\end{split}
\end{equation}
where $d\mathrm{vol}_g=\frac12\omega\wedge\omega$ is the volume form and
\begin{equation*}
	h(u)\coloneqq \frac12(1-|u|^2),\qquad V_J\coloneqq \sum_{k=1}^2[\langle \nabla_{Je_k}u,iu\rangle e_k-\langle \nabla_{e_k}u,iu\rangle Je_k].
\end{equation*}
When integrated, the terms in the second line of \eqref{eq: bradlow decomp} are constant under compactly supported variations, hence solutions of the first order Bradlow system 
\begin{equation}\label{eq: Bradlow eq}
	(\nabla u)^{0,1}_J=0,\quad F_{\nabla;J}^{0,2}=0,\quad \Lambda F_\nabla=h(u).
\end{equation}
minimise $E(u,\nabla)$ under compactly supported perturbations.
Here $\Lambda$ denotes the contraction with the K\"ahler form; see \Cref{sec: bradlow} for a precise definition of the notation used here. 
We call a solution $(u,\nabla)$ to \eqref{eq: Bradlow eq} a Bradlow vortex with respect to the K\"ahler structure $(M,g,J,\omega)$. Such first order equations have to be regarded as direct four-dimensional analogues of the two-dimensional Bogomolnyi equations for planar vortices \cite{Bogomolnyi1976}
\begin{equation}\label{eq: bogo}
	\nabla_{e_1}u+i\nabla_{e_2}u=0,\quad *F_\nabla=h(u).
\end{equation}
The second equation in \eqref{eq: Bradlow eq} implies that $\left(\nabla ^{0,1}_J\right)^2=-iF^{0,2}_{\nabla;J}=0$ and hence $\nabla ^{0,1}_J$ defines a holomorphic structure, for which $u$ is a holomorphic section thanks to the first equation. Consequently, the nodal set is a holomorphic divisor, whose regular part is calibrated by $\omega$, unless $u\equiv 0$.

\medskip

Our main result states that any stable solution in $\R^4$ with quadratic energy growth is a Bradlow vortex.

\begin{theorem}\label{thm: classification}
	Let $(u,\nabla)$ be a stable solution to the self-dual Yang--Mills--Higgs equations in $\R^4$ satisfying 
	\begin{equation}\label{eq: energy bound}
		\int_{B_R}e(u,\nabla)\leq CR^2
	\end{equation}
	for some $C>0$ and all $R\geq 1$. Then, there exists a constant orthogonal complex structure $J$ in $\R^4$ for which $(u,\nabla)$ solves \eqref{eq: Bradlow eq}.
\end{theorem}

By a holomorphic curve in $\C^2$ we mean a closed complex analytic subset of pure complex dimension one.
Our second result serves to complete the picture, by proving that any holomorphic curve with quadratic area growth is the nodal set of a solution to \eqref{eq: Bradlow eq}.

\begin{theorem}\label{thm: existence}
	Let $\Sigma$ be a holomorphic curve in $\R^4$ with respect to some constant orthogonal complex structure $J$ and assume that $\Sigma$ has quadratic area growth. Then, there exists a solution $(u,\nabla)$ to \eqref{eq: Bradlow eq} with quadratic energy growth and satisfying $\{u=0\}=\Sigma$, counting multiplicity.
\end{theorem}

The proof of \Cref{thm: existence} is based on a reduction to a single nonlinear scalar equation depending on the holomorphic function defining the divisor, which is classical \cite{Taubes1980,Bradlow1990}. We prove existence of a solution to this equation using barriers and a compactness argument.

As a corollary to \Cref{thm: classification}, we establish the analogous classification result in $\R^3$.

\begin{corollary}\label{cor: class R3}
	Let $(u,\nabla)$ be a stable solution to the self-dual Yang--Mills--Higgs equations in $\R^3$ satisfying 
	\begin{equation*}
		\int_{B_R}e(u,\nabla)\leq CR
	\end{equation*}
	for some $C>0$. Then $(u,\nabla)$ is a two-dimensional Bogomolnyi vortex, namely there exists an orthonormal frame $\{e_1,e_2,e_3\}$ such that $(u,\nabla)$ solves \eqref{eq: bogo} and
	\begin{equation*}
		\nabla_{e_3}u=0,\quad \iota_{e_3}F_\nabla=0.
	\end{equation*} 
\end{corollary}

We remark that while in $\R^2$ all finite energy solutions are Bogomolnyi vortices \cite{Taubes1980,Taubes1980b}, hence minimising, this is not the case in other ambient manifolds. In closed manifolds this follows from the min-max construction in \cite{Pigati-Stern2021} (see \cite[Remark 2.1]{Cheng2021}). Even in $\R^4$ such equivalence fails: the author together with M.~ del Pino \cite{Badran-delPino2023} constructed critical points whose non-flat nodal set is (asymptotically) contained in a slice $\R^3\times\{0\}\subset\R^4$, and are therefore unstable by \Cref{thm: classification} and a compactness argument. We also note that Cheng \cite{Cheng2021} proved that in highly symmetrical surfaces, such as the round sphere $\Sp^2$ and the flat torus $\mathbf{T}^2$, nonzero stable solutions satisfy the Bogomolnyi equations.

\subsection{Outline of the proof} 

We begin by normalising the self-dual and anti-self-dual parts of
the curvature to obtain two candidate K\"ahler forms,
\begin{equation*}
	\omega_\pm=\sqrt{2}\frac{F_\nabla^\pm}{|F_\nabla^\pm|},
\end{equation*}
defined wherever the corresponding curvature component is nonzero.
These determine orthogonal almost complex structures inducing opposite
orientations, and for which the $(0,2)$ component of $F_\nabla$ automatically vanishes. We associate to them the scalar Bradlow defects
\begin{equation*}
	g_\pm=h-\sqrt{2}|F_\nabla^\pm|
\end{equation*}
Each of these defects solves a PDE, see e.g.~ \eqref{eq: eq for g+}, which shows that the vanishing of either gives constancy of the candidate form $\omega_\pm$ and the conclusion of the Theorem.

This choice between two geometric alternatives has precedents in other
stability classifications. In Micallef's theorem \cite{Micallef1984},
holomorphicity corresponds to constancy of one of the two components
of the Gauss map. Bourguignon--Lawson \cite{Bourguignon-Lawson1981}
establish a self-dual/anti-self-dual dichotomy for stable Yang--Mills
connections on $\Sp^4$ with structure group $SU(2)$,
$SU(3)$, or $U(2)$. Cheng's classification \cite{Cheng2021} similarly uses stability to select between the vortex and anti-vortex equations.

The key observation is that both defects can be controlled through
a single auxiliary quantity,
\begin{equation*}
	w=h-\frac12\operatorname{tr}B
=\min\{g_+,g_-\},
\end{equation*}
where $F$ is the skew-symmetric matrix representing $F_\nabla$ and $B\coloneqq |iF|=\sqrt{-F^2}$ is a symmetric nonnegrative bilinear form. Indeed, representing $F_\nabla=\lambda_1\theta^1\wedge \theta^2+\lambda_2\theta^3\wedge \theta^4$ in an oriented orthonormal coframe, we have
\begin{equation*}
	\frac12\operatorname{tr}|iF|=|\lambda_1|+|\lambda_2|=\max\{|\lambda_1+\lambda_2|,|\lambda_1-\lambda_2|\}=\sqrt2\max\{|F_\nabla^+|,|F_\nabla^-|\}.
\end{equation*}
We will consider two different families of test pairs for the stability inequality
\begin{equation}\label{eq: tests}
	{\bf V}_\alpha=(\nabla_{e_\alpha}u,\iota_{e_\alpha}F_\nabla),\quad {\bf W}_\alpha=(i\nabla_{e_\alpha}u,-B_\alpha),\quad \alpha=1,\dots,4,
\end{equation}
where $B_\alpha$ is the 1-form defined by $B_\alpha(X)\coloneqq |iF|(e_\alpha,X).$
Taking the difference of the quadratic form applied to the test fields and summing in $\alpha$ yields a remainder
\begin{equation*}
	(\|\nabla F\|^2-\|\nabla B\|^2)+4(\nabla^Au)^\dagger(B+iF)(\nabla^Au)
\end{equation*}
where both terms are nonnegative. Interestingly, the first term can be interpreted as a Yang--Mills--Higgs version of the Sternberg--Zumbrun quantity \cite{Sternberg-Zumbrun1998},
\begin{equation*}
	|D^2u|^2-|\nabla|\nabla u||^2
\end{equation*}
which can be used to prove flatness of stable Allen--Cahn level sets in dimensions 2 and 3 \cite{Farina-Sciunzi-Valdinoci2008}. The second term is nonnegative because $B+iF=|iF|+iF$ is positive semidefinite.

\medskip

From this test we deduce the weak equation
\begin{equation*}
	(-\Delta+|u|^2)w=0 \qquad\text{in }\R^4.
\end{equation*}
Since $w^2$ is bounded by a constant multiple of the energy density, the quadratic energy growth and a logarithmic cutoff argument imply $w\equiv0$. Consequently, both defects are nonnegative and their minimum vanishes everywhere.
Stability enters twice: first, we use nonnegativity of the second variation to force suitable test pairs to have vanishing quadratic form in the cutoff limit. Secondly, it supplies the Cauchy--Schwarz inequality for the associated bilinear form, which converts this vanishing into the weak equation for $w$.
Dimension four enters both through the self-dual decomposition of two-forms and through the quadratic growth of the codimension-two energy, which allows the logarithmic cutoff errors to vanish.

\smallskip

Lastly, the test pairs \({\bf W}_\alpha\) in \eqref{eq: tests} admit a natural geometric interpretation. Suppose, for definiteness, that \Cref{thm: classification} selects the branch corresponding to \(J_+\). Its conclusions imply
\begin{equation*}
	\nabla^A_{J_+e_\alpha}u=i\nabla^A_{e_\alpha}u, \quad \iota_{J_+e_\alpha}F_\nabla=-B_\alpha.
\end{equation*}
Thus, setting \(\widetilde e_\alpha:=J_+e_\alpha\), we obtain
\begin{equation*}
	{\bf W}_\alpha
=(\nabla^A_{\widetilde e_\alpha}u,
\iota_{\widetilde e_\alpha}F_\nabla).
\end{equation*}
In other words, the \({\bf W}_\alpha\) are precisely the translation test pairs \({\bf V}_\alpha\) associated with the rotated orthonormal frame \((J_+e_\alpha)\). The advantage of defining them as \({\bf W}_\alpha=(i\nabla^A_{e_\alpha}u,-B_\alpha)\) is that this definition does not require knowing the complex structure in advance.
\section{Setting}
\subsection{The abelian Higgs framework}
We work in $\R^n$ with the Euclidean metric, which we denote $g$. The topological triviality of $\R^n$ ensures that every complex line bundle is trivial and that any unitary connection can be written as $\nabla^A\coloneqq d-iA$ for some one-form $A\in \Gamma(T^*\R^n)$. Hence, we write the energy \eqref{eq: energy} equivalently as 
\begin{equation*}
	E(u,A;\Omega)\coloneqq \frac{1}{2}\int_\Omega e(u,A),\quad e(u,A)\coloneqq  |\nabla^Au|^2+|F_A|^2+h(u)^2.
\end{equation*}
where $u$ is a complex-valued function and $h(u)=\frac12(1-|u|^2)$. We use angled brackets to denote the inner product in $\C$, that is $\langle a,b\rangle\coloneqq \mathrm{Re}(a\overline b)$. We denote a typical pair complex-valued function--real-valued one form by a bold uppercase letter, with components given by the corresponding lowercase and uppercase letters; thus $\b{\Phi}=(\phi,\Phi)$. The only exception is reserved for the background solution, which will be denoted by ${\bf U}=(u,A)$. Note that the pair $(1,0)$ and all its gauge transformed are trivial solutions, called pure gauges.

The second variation of energy is represented by a quadratic form acting on compactly supported pairs 
\begin{equation*}
	Q_{\bf U}[\b\Phi]=\int |\nabla^A\phi-i\Phi u|^2-2\langle\nabla^Au,i\Phi \phi\rangle+ |d\Phi|^2+\langle u,\phi\rangle^2-h(u)|\phi|^2.
\end{equation*}
We say that ${\bf U}$ is stable if 
\begin{equation*}
	Q_{\bf U}[\b\Phi]\geq 0,\quad\text{for every compactly supported pair }\b\Phi.
\end{equation*}
The corresponding second order operator is not elliptic as a consequence of gauge invariance. For this reason, one usually considers a ``gauge corrected'' quadratic form
\begin{equation*}
	\Q_{\bf U}[\b\Phi]\coloneqq \int |\nabla^A\phi|^2+|d\Phi|^2+|d^*\Phi|^2+4\langle i\Phi\cdot \nabla^Au, \phi\rangle+ (|u|^2-h(u))|\phi|^2+|u|^2|\Phi|^2
\end{equation*}
which relates to $Q_{\bf U}[\b\Phi]$ by 
\begin{equation}\label{eq: relation Qs}
	\Q_{\bf U}[\b\Phi]=Q_{\bf U}[\b\Phi]+\int \G_{\bf U}[\b\Phi]^2,\qquad \G_{\bf U}[\b\Phi]\coloneqq d^*\Phi+\langle iu,\phi\rangle.
\end{equation}
The meaning of the operator $G_{\bf U}$ resides in the fact that a pair $\b\Phi$
 satisfying $\G_{\bf U}[\b\Phi]=0$ is $L^2$ orthogonal to the gauge-part of the kernel of the linearised operator. 
 
\smallskip

Consider the bilinear form obtained by polarisation of $Q_{\bf U}$,
\begin{equation*}
	B_{\bf U}[\b\Phi,\b\Psi]\coloneqq \frac{1}{4}\left(Q_{\bf U}[\b\Phi+\b\Psi]-Q_{\bf U}[\b\Phi-\b\Psi]\right).
\end{equation*}
A useful consequence of stability is the Cauchy--Schwarz inequality 
\begin{equation}\label{eq: Cauchy-Schwarz}
	|B_{\bf U}[\b\Phi,\b\Psi]|^2\leq Q_{\bf U}[\b\Phi]Q_{\bf U}[\b\Psi].
\end{equation}
Moreover, we record the simple consequence of \eqref{eq: relation Qs}. If $\b\Psi$ is a compactly supported pair and $G[\b\Phi]=0$, then 
\begin{equation}\label{eq: bilinear pair}
\begin{split}
	B_{\bf U}[\b\Phi,\b\Psi]&=\int \langle \nabla^A\phi,\nabla^A\psi\rangle +d\Phi\cdot d\Psi+d^*\Phi\cdot d^*\Psi\\
	&+\int (|u|^2-h(u))\langle \phi,\psi\rangle +|u|^2\Phi\cdot\Psi\\
	&+\int 2\langle i\Phi\cdot\nabla^Au,\psi\rangle +2\langle i\Psi\cdot\nabla^Au,\phi\rangle.
\end{split}
\end{equation}

\subsection{The Bradlow decomposition on K\"ahler manifolds}\label{sec: bradlow}
On the Euclidean space $(\R^4,g)$ consider an integrable constant orthogonal complex structure $J$ and note that, by constancy, the form $\omega_J(X,Y)\coloneqq g(JX,Y)$
is closed and thus defines a K\"ahler manifold $(\R^4,g,J,\omega_J)$. Consider the orientation naturally induced by the K\"ahler form, so that $\ast \omega_J=\omega_J$. The curvature of the connection $\nabla$ is identified with the real-valued two-form $F_\nabla=i\nabla^2$. Denote the contraction with the K\"ahler form $\Lambda_J\beta\coloneqq \langle \beta,\omega_J\rangle$, for every real two-form $\beta$. We use the complex structure $J$ to split the complexified cotangent bundle into $(0,1)$ and $(1,0)$ forms, and define $\overline \partial^J_Au\coloneqq (\nabla^Au)^{0,1}$. More explicitly,
\begin{equation*}
	\overline \partial^J_Au(X)=\frac12\left(\nabla_Xu+i\nabla_{JX}u\right).
\end{equation*}
Similarly, the space of complexified two-forms splits in its $(0,2)$, $(1,1)$ and $(2,0)$ components:
\begin{equation*}
	F_A=F_{A;J}^{1,1}+F_{A;J}^{0,2}+F_{A;J}^{2,0}.
\end{equation*}
Since $F_A$ is real, the condition $F_{A;J}^{2,0}=0$ is equivalent to $F_A$ being of type $(1,1)$, namely $F_A(JX,JY)=F_A(X,Y)$. For a real $(1,1)$ form, its self-dual part (with respect to the orientation induced by $J$) is 
\begin{equation*}
	F_A^+=\frac{1}{2}(\Lambda_J F_A)\omega_J.
\end{equation*}
Where it causes no confusion, we omit the complex structure $J$ from the notation.
\section{Proof of main result}
\subsection{Classification in $\R^4$}
The first step of the proof is to identify a quantity acting as ``Bradlow defect'', and show that it solves a certain PDE.
Set $j_Au\coloneqq \langle \nabla^Au,iu\rangle$ and $\psi(X,Y)\coloneqq \langle \nabla^A_Xu,i\nabla^A_Yu \rangle $. It is direct to check (see e.g.~ \cite[(2.6)]{Pigati-Stern2021}) that, denoting with $\Delta$ the negative spectrum Hodge laplacian 
\begin{equation*}
	(-\Delta+|u|^2) F_A=-2\psi.
\end{equation*}
Taking the self-dual part we get 
\begin{equation*}
	(-\Delta+|u|^2) F_A^+=-2\psi^+.
\end{equation*}
Now, $F_A^+$ is the candidate K\"ahler form but it is not constant, even in length. On the set $\Omega_+\coloneqq \{|F_A^+|\ne 0\}$, we set 
\begin{equation*}
	\rho_+\coloneqq \sqrt2|F_A^+|,\quad \omega_+\coloneqq \sqrt{2}\frac{F_A^+}{|F_A^+|}
\end{equation*}
so that $|\omega_+|^2\equiv 2$. Such $\omega_+$ defines an orthogonal almost complex structure $J_+$ on $\Omega_+$ obtained by $\omega_+(X,Y)=g(J_+X,Y).$ 
Now, we have 
\begin{equation*}
	\langle \Delta F_A^+,\omega_+\rangle=|u|^2\rho_++2\langle \psi^+,\omega_+\rangle.
\end{equation*}
If we differentiate $F_A^+=\frac12\rho_+\omega_+$ and use that $\langle \nabla_k\omega_+,\omega_+\rangle=0$ and $\langle \Delta\omega_+,\omega_+\rangle=-|\nabla\omega_+|^2$, we get 
\begin{equation*}
	\langle \Delta F_A^+,\omega_+\rangle=\Delta\rho_+-\frac12\rho_+|\nabla \omega_+|^2.
\end{equation*}
Hence 
\begin{equation*}
	(\Delta -|u|^2)\rho_+=2\langle \psi^+,\omega_+\rangle+\frac12\rho_+|\nabla \omega_+|^2.
\end{equation*}
By the Bochner identity we have 
\begin{equation}\label{eq: eq for h}
	(-\Delta +|u|^2)h=|\nabla^Au|^2
\end{equation}
where $h\coloneqq h(u)$. Hence, we find 
\begin{equation*}
	(-\Delta +|u|^2)(h-\rho_+)=|\nabla^Au|^2+2\langle \psi^+,\omega_+\rangle+\frac12\rho_+|\nabla \omega_+|^2
\end{equation*}
Now, a direct calculation shows that 
\begin{equation*}
	2|\overline\partial_A^{J_+}u|^2=|\nabla^Au|^2+2\langle \psi^+,\omega_+\rangle
\end{equation*}
hence, setting $g_+\coloneqq h-\rho_+$
\begin{equation}\label{eq: eq for g+}
	(-\Delta +|u|^2)g_+=2|\overline\partial_A^{J_+}u|^2+\frac12\rho_+|\nabla \omega_+|^2
\end{equation}
The same equation follows identically for $g_-\coloneqq h-\rho_-$ taking the anti-self-dual part. 

\begin{lemma}\label{lem: g vanishing implies Bradlow}
If $g_+\equiv 0$, then $(u,A)$ is a Bradlow vortex with respect to the K\"ahler structure $(\R^4,g,J_+,\omega_+)$. Similarly for $g_-$. If $(u,A)$ is a pure gauge, it is a Bradlow vortex with respect to any K\"ahler structure.
\end{lemma}
\begin{proof}
The assumption implies that $h= \rho_+\geq 0$ and maximum principle, together with \eqref{eq: eq for h} implies that either $h= 0$ or $h>0$. In the first case, $|u|= 1$ and $|\nabla^Au|= 0$ by \eqref{eq: eq for h}, so the solution is a pure gauge and thus a Bradlow vortex with respect to any K\"ahler structure. If $\rho_+=h>0$, then $\Omega_+=\R^4$ and we proceed as follows.
	By \eqref{eq: eq for g+} we have 
\begin{equation}\label{eq: Brad1}
	\overline\partial_A^{J_+}u=0
\end{equation}
and $\omega_+$ constant on $\R^4$, which makes $(R^4,g,J_+,\omega_+)$ K\"ahler. 
Then, we have 
\begin{equation}\label{eq: Brad2}
	F_A^+=\frac12\rho_+\omega_+=\frac{h}{2}\omega_+.
\end{equation}
Thus, putting \eqref{eq: Brad1} and \eqref{eq: Brad2} together we obtain that $(u,A)$ is a Bradlow vortex with respect to the structure $(\R^4,g,J_+,\omega_+).$ The proof is analogous if $g_-\equiv 0$ replacing the K\"ahler structure with $(\R^4,g,J_-,\omega_-)$.
\end{proof}

\begin{proof}[Proof of \Cref{thm: classification}] Let $\eta\in C^\infty_c(\R^4)$. Consider the pairs
\begin{equation*}
	{\bf V}_\alpha\coloneqq (\iota_{e_\alpha}\nabla^Au,\iota_{e_\alpha}dA),\quad \alpha=1,2,3,4.
\end{equation*} 
To ease the notation we omit the subscript ${\bf U}$ from ${\bf Q}_{\bf U}$, $Q_{\bf U}$ and $G_{\bf U}$.
A direct calculation shows that
\begin{equation*}
	\Q[\eta{\bf V}_\alpha]=\int L_{\bf U}[{\bf V}_\alpha]\cdot \eta^2{\bf V}_\alpha+\int |d\eta|^2|{\bf V}_\alpha|^2=\int |d\eta|^2|{\bf V}_\alpha|^2
\end{equation*}
where $L_{\bf U}$ is the linearised operator about ${\bf U}$ associated to the gauge-corrected quadratic form. Here we used that $L_{\bf U}[{\bf V}_\alpha]=0$ by direct differentiation. Hence, we get
\begin{equation}\label{eq: first test}
	\sum_{\alpha=1}^4\Q[\eta{\bf V}_\alpha]=\int |d\eta|^2(|\nabla^Au|^2+2|dA|^2)
\end{equation}
Next, let $F=(F_{jk})_{jk}$ be the real-valued skew-symmetric matrix whose entries coincide with the ones of the two form $F_A$. Skew symmetry implies that the matrix 
\begin{equation*}
	B\coloneqq \sqrt {-F^2}=\sqrt{F^\top F}
\end{equation*}
is well defined, symmetric and positive semidefinite. Moreover, 
\begin{equation}\label{eq: B-F estimate}
	\|B\|^2=\|F\|^2=2|F_A|^2,\quad \|\nabla B\|^2\leq \|\nabla F\|^2.
\end{equation}
where $\|\cdot\|$ denotes the Hilbert--Schmidt norm.
Hence, we define the form $B_\alpha\coloneqq B_{\alpha j}dx^j$ and 
\begin{equation*}
	{\bf W}_\alpha=(i\nabla^A_\alpha u,-B_\alpha).
\end{equation*}
Note that ${\bf W}_\alpha$ is in principle only locally Lipschitz, since $B_\alpha$ is. Nevertheless the stability inequality can be extended by density to all compactly supported $H^1$ pairs.
Now we compute 
\begin{equation}\label{eq: test difference}
\begin{split}
		\sum_{\alpha=1}^4\left(\Q[\eta {\bf W}_\alpha]-\Q[\eta {\bf V}_\alpha]\right)=&-\int \eta^2(\|\nabla F\|^2- \|\nabla B\|^2)\\
	&-4\int \eta^2\sum_{\alpha,\beta=1}^4B_{\alpha\beta}\langle \nabla^A_\alpha u,\nabla^A_\beta u\rangle \\
	&+4\int \eta^2\sum_{\alpha,\beta=1}^4F_{\alpha\beta}\langle i\nabla^A_\alpha u,\nabla^A_\beta u\rangle
\end{split}
\end{equation}
The last two terms can be interpreted in the following way: define the complex column vector $D^Au=(\nabla^A_1u,\dots,\nabla^A_4u)^\top$. Then 
\begin{equation*}
	-\sum_{\alpha,\beta=1}^4\left[B_{\alpha\beta}\langle \nabla^A_\alpha u,\nabla^A_\beta u\rangle-F_{\alpha\beta}\langle i\nabla^A_\alpha u,\nabla^A_\beta u\rangle\right]=-(D^Au)^\dagger (B+iF)(D^Au)
\end{equation*}
It is directly checkable that $B=|iF|$, which ensures that $B+iF$ is positive semidefinite. Using this fact, \eqref{eq: first test} and \eqref{eq: test difference} we get to 
\begin{equation}\label{eq: test}
\begin{split}
	\sum_{\alpha=1}^4\Q[\eta {\bf W}_\alpha]+\int\eta^2(\|\nabla F\|^2- &\|\nabla B\|^2)+4\int\eta^2(D^Au)^\dagger (B+iF)(D^Au)\\
	&=\int |d\eta|^2(|\nabla^Au|^2+2|dA|^2)
\end{split}
\end{equation}
By stability, \eqref{eq: B-F estimate} and semipositivity of $B+iF=|iF|+iF$, the left-hand side is nonnegative. 
For any $R>2$, let $\eta_R$ be a logarithmic-cutoff function with the following properties: $\eta_R\equiv 1$ in $B_R$, $\eta_R\equiv 0$ in $\R^4\setminus B_{R_2}$ and 
\begin{equation*}
	|d\eta_R|(x)\leq \frac{C}{|x|\log R}.
\end{equation*}
Using the energy bound \eqref{eq: energy bound}
implies vanishing of the right-hand side in \eqref{eq: test}, hence the vanishing of each term of the left-hand side. In particular, by \eqref{eq: relation Qs} we have 
\begin{equation}\label{eq: split stability consequence}
	Q[\eta_R{\bf W}_\alpha]\to 0,\quad \text{and}\quad \int G[\eta_R{\bf W}_\alpha]^2\to 0.
\end{equation} 
By the Cauchy--Schwarz inequality \eqref{eq: Cauchy-Schwarz} we have, for any compactly supported pair $\b{\Psi}$ and $R$ large enough,
\begin{equation*}
	0\leq |B({\bf W}_\alpha,\b{\Psi})|^2=|B(\eta_R{\bf W}_\alpha,\b{\Psi})|^2\leq Q(\eta_R{\bf W}_\alpha)Q(\b{\Psi})\to 0
\end{equation*}
thus $B({\bf W}_\alpha,\b{\Psi})=0$. By \eqref{eq: split stability consequence} we have that $G[{\bf W}_\alpha]=0$ a.e.; testing \eqref{eq: bilinear pair} with $\Phi={\bf W}_\alpha$ and $\b{\Psi}=(0,\Psi)$ for any compactly supported one-form $\Psi$ gives 
\begin{equation*}
	0=B({\bf W}_\alpha,\b{\Psi})=-\int dB_\alpha\cdot d\Psi+d^*B_\alpha\cdot d^*\Psi+|u|^2B_\alpha\cdot\Psi-2\Psi\cdot\langle i\nabla^Au,i\nabla^A_\alpha u\rangle
\end{equation*}
which is the weak formulation of
\begin{equation*}
	(-\Delta+|u|^2)B_{\alpha\beta}=2\langle \nabla^A_\alpha u,\nabla^A_\beta u\rangle.
\end{equation*}
Taking the trace and using \eqref{eq: eq for h}, this shows that setting $w\coloneqq h(u)-\frac{1}{2}\mathrm{tr}B$, the equation
\begin{equation*}
		(-\Delta+|u|^2)w=0
\end{equation*}
is solved weakly. Testing with $\eta_R^2w$ and integrating by parts we find 
\begin{equation*}
	\int |d(\eta_Rw)|^2+\eta_R^2|u|^2|w|^2=\int |d\eta_R|^2|w|^2.
\end{equation*}
Using that $|w|^2\leq 2h(u)^2+4|F_A|^2\leq 4e(u,A)$, the same log-cutoff argument implies that $dw=0$ and $|u|w=0$. Since the $u\equiv 0$ solution is excluded by energy growth, this implies $w=0$ and hence
\begin{equation}\label{eq: h minus trace B}
	h(u)=\frac{1}{2}\mathrm{tr}B.
\end{equation}
Lastly, choose an oriented orthonormal coframe $\theta^1,\dots,\theta^4$ such that $F_A=\lambda_1\theta^1\wedge \theta^2+\lambda_2\theta^3\wedge \theta^4$. Then 
\begin{equation*}
	\rho_\pm=|\lambda_1\pm \lambda_2|,\quad \frac{1}{2}\mathrm{tr}B=|\lambda_1|+|\lambda_2|=\max\{\rho_-,\rho_+\}
\end{equation*}
Hence, by \eqref{eq: h minus trace B},
\begin{equation}\label{eq: eq for h-maxrho}
	0=h-\frac{1}{2}\mathrm{tr}B=h-\max\{\rho_-,\rho_+\}=\min\{g_-,g_+\}.
\end{equation}
Equation \eqref{eq: eq for h-maxrho} shows that $g_{\pm}\geq h-\max\{\rho_-,\rho_+\}= 0$ and that one of the two vanishes at some point $x_0$. We use again the maximum principle in \eqref{eq: eq for h} to infer that either $h=0$ or $h>0$. In the first case, the solution is a pure gauge and there is nothing to prove, in the second case suppose without loss of generality that $h(x_0)=\rho_+(x_0)>0$. Then, maximum principle and \eqref{eq: eq for g+} imply that $g_+$ vanishes on the connected component of $\Omega_+$ containing $x_0$. Such connected component has to be $\R^4$; otherwise pick a point $x_1$ in its boundary. By continuity, there we would have $h(x_1)=\rho_+(x_1)=0$, a contradiction. So $g_+=0$ on $\R^4$ and applying \Cref{lem: g vanishing implies Bradlow} we conclude the proof. 
\end{proof}

\subsection{Solutions with prescribed nodal set}
Let $\Sigma\subset\R^4$ be a holomorphic curve with respect to some complex structure $J.$ Here we construct a solution to the first order equations \eqref{eq: Bradlow eq} such that $\{u=0\}=\Sigma$, with the prescribed multiplicities.

\begin{proof}[Proof of \Cref{thm: existence}]
	Let $f\colon \C^2\simeq \R^4\to\C$ be the polynomial cutting out $\Sigma$, which exists by \cite[Theorem D]{Stolzenberg1966}. Here we fixed the complex structure $J$ that makes $f$ holomorphic, and we omit it in the notations below. We can suppose $m\coloneqq \deg(f)\geq 1$. As in \cite{Taubes1980}, we encode the singularity model of the sought solution via $f$ and look for a pair of the form 
	\begin{equation*}
		u=e^{-\frac\phi2}f,\quad A=\frac12d^c\phi
	\end{equation*}
	where $d^c\phi(X)\coloneqq -d\phi(JX)$. We have $A^{0,1}=\frac{i}2\overline\partial\phi$ and 
	\begin{equation*}
		\overline\partial_Au=\overline\partial u-iA^{0,1}u=-\frac12e^{-\frac\phi2}f\overline\partial\phi+e^{-\frac\phi2}\overline\partial f-i\frac{i}2e^{-\frac\phi2}f\overline\partial\phi=0,
	\end{equation*}
	where we used holomorphicity $\overline\partial f=0$.
	Moreover, 
	\begin{equation*}
		F_A=dA=\frac12dd^c\phi=i\partial\overline\partial\phi=i\sum_{j,k=1}^2\phi_{z_j\overline z_k}dz_j\wedge d\overline z_k
	\end{equation*}
	so $F_A$ is of type $(1,1)$ and hence $F_A^{0,2}=0.$ Contraction with the K\"ahler form $\omega=\frac{i}2\sum_{j=1}^2dz_j\wedge d\overline z_j$ gives 
	\begin{equation*}
		\Lambda F_A=\frac12\Delta\phi.
	\end{equation*}
	Hence, \eqref{eq: Bradlow eq} is satisfied if we solve the reduced equation 
	\begin{equation}\label{eq: reduced eq for phi}
		\Delta\phi=1-|f|^2e^{-\phi}\quad \text{in }\R^4.
	\end{equation}
	We use the subsolution-supersolution method. In what follows, we use multiple times the fact that for a holomorphic function $g$
	\begin{equation}\label{eq: holo property}
		\Delta |g|^2=4\sum_{j}|\partial_{z_j}g|^2
	\end{equation}
	Let $L(v)\coloneqq -\Delta v+1-|f|^2e^{-v}$ and let
	\begin{equation*}
		S\coloneqq \sum_{|\alpha|\leq m}\frac{(4m)^{|\alpha|}}{\alpha!}|\partial^\alpha f|^2,\quad \Phi=\log S.
	\end{equation*}
	Since some derivative of $f$ of order $m$ are nonzero, $S>0$. With this choice of coefficients, using \eqref{eq: holo property},
	\begin{equation*}
		\Delta S=4\sum_{\alpha,j}c_\alpha|\partial^{\alpha+e_j}f|^2=\sum_{1\leq |\beta|\leq m}\frac{|\beta|}{m}c_\beta|\partial^\beta f|^2\leq S-|f|^2.
	\end{equation*}
	Cauchy--Schwarz, \eqref{eq: holo property} and holomorphicity of $f$ give $|\nabla S|^2\leq S\Delta S$, which is equivalent to $\Delta\log S\geq 0$. Then
	\begin{equation*}
		0\leq \Delta\Phi=\frac{\Delta S}{S}-\frac{|\nabla S|^2}{S^2}\leq \frac{\Delta S}{S}\leq 1-\frac{|f|^2}{S}\leq 1-|f|^2e^{-\Phi}
	\end{equation*}
	Thus $L[\Phi]\geq 0$. On a fixed ball $B_R$, a subsolution is provided by the parabola $P_R(z)=\frac{|z|^2}{8}-C_R$, since 
	\begin{equation*}
		L[P_R]=-|f|^2e^{-P_R}\leq 0
	\end{equation*}
	and choosing the constant $C_R>0$ large enough so that $P\leq\Phi$ in $B_R$. The subsolution-supersolution method gives a smooth solution $P_R\leq \phi_R\leq \Phi$ in $B_R$ with $\phi_R=\Phi$ on $\partial B_R$. Now, the function $v\coloneqq \log|f|^2-\phi_R$ defined on $B_R\setminus\{f=0\}$ solves $\Delta v=e^v-1$, it is non-positive on $\partial B_R$ and tends to $-\infty$ as it approaches $\{f=0\}$. By maximum principle it cannot have a positive interior maximum, so 
	\begin{equation*}
		\log|f|^2\leq \phi_R\leq \Phi,\quad 0\leq \Delta\phi_R\leq 1.
	\end{equation*}
	So $w_R\coloneqq \Phi-\phi_R\geq 0$, $|\Delta w_R|\leq 1$ and at any $p\notin\{f=0\}$ we have $w_R(p)\leq \Phi(p)-\log|f(p)|^2$. Harnack inequality and elliptic estimates give bounds on every compact set, uniformly in $R$, and a diagonal argument provides a global solution $\phi$ satisfying $\log|f|^2\leq \phi\leq \Phi.$
	
	We are only left to prove the quadratic energy growth. Let $c_R\coloneqq 2m\log R$. We claim that 
	\begin{equation}\label{eq: L2 bound}
		\|\phi-c_R\|_{L^2(B_{2R})}\leq CR^2
	\end{equation}
	for some $C>0$. Indeed consider $p_R(z_1,z_2)\coloneqq R^{-m}f(Rz_1,Rz_2)$ and assume without loss of generality that it is a monic polynomial in $z_1$ whose coefficients are bounded uniformly in $z_2\leq 2$. This implies that $\|\log|p_R|\|_{L^2(B_2)}\leq C$. The bound $\log|f|^2\leq \phi\leq \log S$ and scaling yield \eqref{eq: L2 bound}.
	
	Let now $q=\Delta\phi=1-|u|^2=2h$ and consider a cutoff $\chi_R$ supported in $B_{2R}$, identically 1 in $B_R$ and satisfying $|\Delta\chi_R|\leq CR^{-2}$. By \eqref{eq: L2 bound}, 
	\begin{equation*}
		\int_{B_R}q\leq \int_{B_{2R}}\chi_R\Delta\phi=\int_{B_{2R}}(\phi-c_R)\Delta\chi_R\leq CR^2.
	\end{equation*}
	Since $q\leq 1$, we get 
	\begin{equation*}
		\int_{B_R}(1-|u|^2)^2\leq \int_{B_R}q\leq CR^2.
	\end{equation*}
	Using that $|F_A|\leq C|D^2\phi|$ and interior elliptic estimates
	\begin{equation*}
		\int_{B_R}|F_A|^2\leq C\int_{B_R}|D^2\phi|^2\leq C\int_{B_{2R}}q^2+CR^{-4}\int_{B_{2R}}|\phi-c_R|^2 \leq CR^2.
	\end{equation*}
	Lastly, by \eqref{eq: eq for h} we have $2|\nabla^Au|^2=\Delta|u|^2+q|u|^2$ and hence 
	\begin{equation*}
		2\int_{B_R}|\nabla^Au|^2\leq \int_{B_{2R}}(|u|^2-1)\Delta\chi_R+\int_{B_{2R}}\chi_Rq|u|^2\leq CR^2
	\end{equation*}
	concluding the proof.
\end{proof}

\subsection{Classification in $\R^3$}
Here we prove \Cref{cor: class R3}. The proof follows by extending the solution to $\R^4$ trivially and applying \Cref{thm: classification}.

\begin{proof}[Proof of \Cref{cor: class R3}] We pull the pair back to $\R^4=\R^3\times \R$ and, with a little abuse of notation, we still denote it with $(u,A)$. Stability passes to the extension and linear energy bounds become quadratic; thus we can apply \Cref{thm: classification}, which supplies a constant complex structure $J$ for which $(u,A)$ solves \eqref{eq: Bradlow eq}. Pick a frame $\{e_1,e_2,e_3,e_4\}$ such that $e_2=Je_1$ and $e_4=Je_3$, with $e_4$ the added direction. By the Bradlow equation for $u$
\begin{equation*}
	0=\nabla^A_{e_3}u+i\nabla^A_{e_4}u=\nabla^A_{e_3}u
\end{equation*}
while using that $F_A$ is of type $(1,1)$ we get
\begin{equation*}
	\iota_{e_3}F_A=\iota_{e_4}F_A\circ J=0.
\end{equation*}
Thus the pair $(u,A)$ is invariant in the $e_3$ direction, up to gauge. The validity of \eqref{eq: bogo} in the plane spanned by $e_1$ and $e_2$ follows directly.
\end{proof}

\subsection*{Acknowledgments} The author was supported by the European Research Council under Grant Agreements No 948029 (StableIF) and No 101165368 (MAGNETIC).

\subsection*{On the use of AI} The validity of the Cauchy--Schwarz inequality for stable operators \eqref{eq: Cauchy-Schwarz}, as well as its use in the proof of \Cref{thm: classification}, was suggested to the author by ChatGPT 5.6 Sol. Apart from this, ChatGPT was used solely as an auxiliary tool for language editing, improving the presentation of the manuscript, and performing basic checks. This paper does not contain AI generated text.

\bibliography{Bib}
\bibliographystyle{siam}

\end{document}